\documentclass[10pt,a4paper]{article}

\usepackage[utf8]{inputenc}
\usepackage[T1]{fontenc}

\usepackage{lmodern}

\usepackage[english]{babel}

\usepackage{amssymb,amsmath,amsfonts,amssymb,amsthm}
\usepackage{graphics,graphicx,color}

\usepackage{anysize}
\marginsize{3cm}{3cm}{2cm}{2cm}

\usepackage{enumerate}

\newcommand{\R}{{\mathbb R}}

\newcommand{\Z}{{\mathbb Z}}
\newcommand{\N}{{\mathbb N}}
\newcommand{\eps}{\varepsilon}
\newcommand{\la}{\left\langle}
\newcommand{\ra}{\right\rangle}

\DeclareMathOperator{\supp}{supp}

\DeclareMathOperator{\Log}{Log}
\DeclareMathOperator{\Lip}{Lip}

\DeclareMathOperator{\spec}{spec}
\DeclareMathOperator{\real}{Re}

\title{Backward uniqueness for parabolic operators with non-Lipschitz coefficients}
\date{\today}
\author{Daniele Del Santo, Christian J\"ah and Marius Paicu}

\newtheorem{defi}{Definition}
\newtheorem{rem}{Remark}
\newtheorem{lem} {Lemma}
\newtheorem{prop}{Proposition}
\newtheorem{thm}{Theorem}
\newtheorem{cor}{Corollary}

\newtheorem{ex}{Example}

\begin{document}

\maketitle
  
\begin{abstract}
In this paper we study the backward uniqueness for parabolic equations with non-Lipschitz coefficients in time and space. 
The result presented here improves an old uniqueness theorem due to Lions and Malgrange [Math. Scand. {\bf 8} (1960), 277--286] and some more recent 
results of Del Santo and Prizzi [J. Math. Pures Appl. {\bf 84} (2005), 471--491; Ann. Mat. Pura Appl., to appear].
\\[0.3 cm]
\end{abstract}

%%%%%%%%%%%%%%%%
\section{Introduction}%%%%%%
%%%%%%%%%%%%%%%%%%

The question of uniqueness and non-uniqueness for solutions of partial differential equations has a fairly long history, starting form the classical works of Holmgren and Carleman. A good and rather complete survey about the results on this topic, until the early 1980's, can be found in the  book of Zuily \cite{Zuily}. 

In this paper we are interested in a particular class of parabolic operators for which we consider the uniqueness property, backwards in time. Uniqueness  for smooth solutions of parabolic and backward parabolic operators is not trivial. In \cite{Tychonov} Tychonoff showed that a solution $u \in C^\infty(\R_t \times \R^n_x)$ of the Cauchy problem \begin{equation} \label{CPT}
\left\{ \begin{array}{ll}
\partial_t u - \Delta_x  u = 0, & (t,x) \in \R_t \times \R^n_x \\
u(0,x) = 0 & x \in \R^n_x,
\end{array}
\right.
\end{equation} not necessarily vanishes.
In particular, the example given by Tychonoff is such that the solution $u(t,x)$ to (\ref{CPT}) satisfies 
\begin{equation}
\sup_{x\in {\mathbb R}_x^n}\; (\max_{t\in [-T,T]} |u(t,x)| e^{-a|x|^2} )=+\infty,
\label{exty}
\end{equation} 
for all $a>0$. On the other hand Tychonoff proved that  uniqueness to (\ref {CPT}) can be obtained, 
 for example, if  one imposes $ \max_{t\in [-T,T]} |u(t,x)| \leq  Ce^{a|x|^2}$, for some $C,\ a > 0$. 
 Other interesting examples of non-uniqueness for \eqref{CPT}, under particular assumptions, can e.g. be found in \cite{Rosenbloom}.

Here we  consider   the {\it backward parabolic} operator \begin{equation} \label{BW-Op}
Pu = \partial_t u + \sum\limits_{j,k=1}^n \partial_{x_j}(a_{jk}(t,x)\partial_{x_k}u) + c(t,x)u,
\end{equation} defined on the strip $[0,T] \times \R^n_x$; all the coefficients are supposed to be measurable and bounded; the 0-order coefficient $c(t,x)$  is allowed to be complex valued and we assume that the matrix $(a_{jk}(t,x))_{j,k=1}^n$ is real and symmetric for all $(t,x) \in [0,T] \times \R^n_x$ and that there exists an $a_0 \in (0,1]$ such that, for all $(t,x,\xi) \in [0,T] \times \R^n_x \times \R^n_\xi$, 
\begin{equation}
\label{elliptic}
\sum\limits_{j,k=1}^n a_{jk}(t,x)\xi_j\xi_k \geq a_0 |\xi|^2.
\end{equation} 

Under {\it uniqueness property in $\mathcal H$} we will mean the following: let $\mathcal H$ be a  space of functions (in which it makes sense to look for solutions $u$ of  the equation $Pu=0$). Then we say that the operator $P$ has the uniqueness property in $\mathcal H$ if, whenever $u \in \mathcal H$, $Pu=0$ on $[0,T]\times \R^n_x$ and $u(0,x)=0$ in $\R^n_x$, then $u \equiv 0$ in $[0,T] \times \R^n_x$.

In \cite{Lions1960} Lions and Malgrange  proved
the uniqueness property for \eqref{BW-Op} in the space 
\begin{equation}\label{defH}
\mathcal H := L^2([0,T], H^2(\R^n_x)) \cap H^1([0,T], L^2(\R^n_x)),
\end{equation} 
(note that this choice for $\mathcal H$ excludes the pathological situation of (\ref{exty}))
under the assumption that, for all $j,k = 1, \dots, n$,
\begin{equation*}
a_{jk}(t,x) \in \Lip([0,T],L^\infty(\R^n_x)).
\end{equation*}

An example of Miller in \cite{Miller1973} showed that the regularity of the coefficients $a_{jk}$ with respect to $t$ should be taken under consideration, if one wants to have uniqueness in $\mathcal H$. In particular he constructed a nontrivial solution to the  Cauchy problem for (\ref{BW-Op})  with $0$ initial data, for an operator having the coefficients $a_{jk}$ in  $C^\alpha([0,T],C_b^\infty(\R^n_x))$, for all $0<\alpha < \frac{1}{6}$. 

The example of Miller was considerably improved by Mandache in \cite{Mandache}, in the following way: consider a modulus of continuity $\mu$ which does not satisfy the Osgood condition, i.e.
\begin{equation*} 
\int\limits_0^1 \frac{1}{\mu(s)} ds< +\infty,
\end{equation*} 
then it is possible to  construct an operator of type (\ref{BW-Op}) having the regularity with respect to $t$ of the coefficients of the principal part ruled by $\mu$, such that this operator does not have the uniqueness property in $\mathcal H$.

In \cite{DSP} Del Santo and Prizzi proved uniqueness for \eqref{BW-Op} in $\mathcal H$, under the condition that, for all $ j,k = 1, \dots, n$,
\begin{equation*}
a_{jk}(t,x) \in C^\mu([0,T], L^\infty(\R^n_x)) \cap L^\infty([0,T],C^2(\R^n_x)), 
\end{equation*} 
and with the modulus of continuity $\mu$ satisfying  the Osgood condition 
\begin{equation} 
\label{Osgoodintr}
\int\limits_0^1 \frac{1}{\mu(s)} ds= +\infty.
\end{equation} 

If the result in \cite{DSP} was completely satisfactory from the point of view of the regularity with respect to $t$, the same cannot be said for the regularity with respect to the space variables:  the $C^2$ regularity with respect to $x$ was a consequence of a difficulty in obtaining the Carleman estimate from which the uniqueness was deduced. 

In \cite{DS2012} Del Santo made the technique used in \cite{DSP} more effective by using a theorem of Coifman and Meyer (\cite[Th. 35]{CM1978}, see also \cite[Par. 3.6]{Tay}) and he could lower the regularity assumption in $x$ from $C^2$ to $C^{1+\eps}$ for an arbitrary small $\eps>0$. 

Refining this approach Del Santo and Prizzi got in \cite{DSP2012} the uniqueness property for \eqref{BW-Op} with the  coefficients of the principal part\begin{equation*}
a_{jk} \in C^\mu([0,T],L^\infty(\R_x^n)) \cap L^\infty([0,T],\Lip(\R_x^n)).
\end{equation*}

In the present  paper we will lower the regularity assumption for the coefficients of the principal part with respect to the space variables, going beyond the Lipschitz-continuity. The regularity with respect to $x$ will be controlled by a modulus of continuity linked to the Osgood modulus of continuity with respect to $t$. More precisely we will prove that the uniqueness property in $\cal H$ for \eqref{BW-Op} holds for principal part coefficients 
\begin{equation*}
a_{jk} \in C^\mu([0,T],L^\infty(\R_x^n)) \cap L^\infty([0,T],C^\omega(\R_x^n)),
\end{equation*} 
where $\mu$ satisfies \eqref{Osgoodintr} and $\omega(s) = \sqrt{\mu(s^2)}$. 
The proof of this uniqueness result will use the Littlewood-Paley theory and the Bony's paraproduct and will be obtained exploiting a Carleman estimate. The Carleman estimate will be proved in $H^{-s}$ with $s \in (0,1)$ while the weight function in the Carleman estimate will be the same as that in \cite{Tarama}. 

The paper is organized as follows. First we state the  uniqueness results and we give some remarks. Then we introduce the Littlewood-Paley theory and Bony's paraproduct. These tools are used in obtaining some estimates, presented in Subsection \ref{remainderest}. Finally, Section \ref{SecCarleman}  is devoted to the proof of the Carleman estimate needed to deduce our uniqueness theorem.

%%%%%%%%%%%%%%%%%%
\section{The uniqueness result}%%%
%%%%%%%%%%%%%%%%%%
\label{par1}
\begin{defi}
A continuous function $\mu: [0,1] \rightarrow \R$ is called \textit{modulus of continuity} if it is strictly increasing, concave and  satisfies $\mu(0)=0$.
 \end{defi}

\begin{rem}
The concavity of the modulus of continuity has a list of simple consequences: for all $s \in [0,1]$ we have $\mu(s) \geq \mu(1)s$,
the function $s \mapsto \mu(s)/s$ is decreasing on $(0,1]$, the limit $\lim_{s \rightarrow 0+} \mu(s)/s$ exists,
the function $\sigma \mapsto \mu(1/\sigma)/(1/\sigma)$ is increasing on $[1,+\infty)$
and the function $\sigma \mapsto 1/(\sigma^2 \mu(1/\sigma))$ is decreasing on $[1,+\infty)$.
Moreover, there exists a constant $C>0$ such that 
\begin{equation}
\mu(2s)\leq C\mu(s).
\label{duplicate}
\end{equation}
\end{rem}

\begin{defi}
Let $\Omega$ be a  convex set  in  $\R^n$ and $f : \Omega \rightarrow \mathcal B$, where $\mathcal B$ is a Banach space. We will say that $f$ belongs to $C^\mu(\Omega,\mathcal B)$ if $f$ is bounded and it satisfies 
\begin{equation*}
\sup\limits_{\substack{0<|t-s|<1 \\t,s \in \Omega}} \frac{\|f(t)-f(s)\|_{\mathcal B}}{\mu(|t-s|)} < +\infty.
\end{equation*} 
For $f\in C^\mu(\Omega,\mathcal B)$ we set 
$$
\|f\|_{C^\mu(\Omega,\mathcal B)}=\|f\|_{L^\infty(\Omega,\mathcal B)}+ \sup\limits_{\substack{0<|t-s|<1 \\t,s \in \Omega}} \frac{\|f(t)-f(s)\|_{\mathcal B}}{\mu(|t-s|)}.
$$
In case of no ambiguity we will omit the space $\mathcal B$ from the notation.
\end{defi}

\begin{defi} 
We will say that a  modulus of continuity $\mu$ satisfies the Osgood condition if
\begin{equation} \label{Osgood}
\int\limits_0^1 \frac{1}{\mu(s)} ds= +\infty.
\end{equation}
 \end{defi}

\begin{ex} A simple example of a modulus of continuity is $\mu(s) = s^\alpha$, for $\alpha\in (0,1]$. If $\alpha \in (0,1)$ (H\"older-continuity) $\mu$ does not satisfies the Osgood condition, while if $\alpha=1$ (Lipschitz-continuity) $\mu$ satisfies the Osgood condition. Similarly  $\mu(s) = s(1+ |\log(s)|)^\alpha$, for  $\alpha > 0$, ($\Log^\alpha$-Lipschitz-continuity) satisfies the Osgood condition if and only if $\alpha \leq 1$. \end{ex}

Now we state our main uniqueness result.

\begin{thm} \label{Uniqueness} Let $\mu$ and $\omega$ be two moduli of continuity such that $\omega(s)= \sqrt{\mu(s^2)}$.
Suppose that $\mu$ satisfies the Osgood condition. Suppose moreover that 
\begin{itemize}
\item []there exists a constant $C>0$ such that
 \begin{equation} 
 \label{Dini1}
\int\limits_0^h \frac{\omega(t)}{t} dt \leq C \omega(h);
\end{equation} 
\item []there exists a constant $C>0$ such that, for all $1\leq p\leq q-1$,
\begin{equation} \label{TechCond1}
\frac{\omega(2^{-q})}{\omega(2^{-p})} \leq C \omega(2^{p-q});
\end{equation} 
for all $s\in (0,1)$,
\begin{equation}
\label{TechCond2}
\sum_{k= 0}^{+\infty} 2^{(1-s)k}\omega (2^{-k})<+\infty.
\end{equation}
\end{itemize}
Assume that, for all $j,k = 1, \dots, n$, 
\begin{equation*}
a_{jk} \in C^\mu([0,T],L^\infty(\R_x^n)) \cap L^\infty([0,T],C^\omega(\R_x^n)). 
\end{equation*}

Then the operator $P$ has the uniqueness property in $\mathcal H$, where $P$ and $\mathcal H$ are defined in \eqref{BW-Op} and \eqref{defH} respectively.
\end{thm}
\begin{rem} We don't know at the present stage whether the conditions \eqref{TechCond1}  and \eqref{TechCond2} are purely technical or can be removed. They are necessary to the proof of some auxiliary remainder estimates (see  Section \ref{remainderest}, Lemma \ref{lemma}). Let us remark that \eqref{TechCond2} is implied by the following: for all $\sigma\in (0,1)$, there exists $\delta_{\sigma}\in (0,1)$ and $c,\ C>0$ such that, for all $s\in [0,\delta_{\sigma}]$, we have $cs\leq \omega(s)\leq C s^{\sigma}$. \end{rem}

%{\color{red} \bf Should we in the last line add a constant on the RHS? Since the modulus of continuity could be $\mu(s)=C s$, $C>0$.}

\begin{rem} It would be possible to prove uniqueness for an operator with terms of order one, i.e. for \begin{equation*}
\tilde P = \partial_t + \sum\limits_{j,k=1}^n \partial_{x_j}(a_{jk}(t,x)\partial_{x_k}) + \sum\limits_{k=1}^n b_k(t,x)\partial_{x_k} + c(t,x),
\end{equation*} assuming that $b_k(t,x)$ are $L^\infty([0,T], C^\sigma(\R_x^n))$ for some $\sigma >0$. This is due to the fact that the Carleman estimate, which we are able to prove, is in $H^{-s}$ with $s \in (0,1)$. In \cite{DSP} and  \cite{DS2012} the Carleman estimate was proved in $L^2$ and this fact allowed to consider the coefficients $b_k(t,x)$ under no hypotheses on $b_k(t,x)$, apart boundedness. \end{rem}

\begin{ex} 
A simple example of moduli of continuity $\mu$ and $\omega$ satisfying the hypotheses of Theorem \ref{Uniqueness} is $\mu(s)= s(1+|\log(s)|)$ and $\omega(s) = s\sqrt{1+|\log(s)|}$.  
\end{ex}

%%%%%%%%%%%%%%%%%%%%%%%%%%%%%%%%
\section{Littlewood-Paley theory  and Bony's paraproduct}%%%
%%%%%%%%%%%%%%%%%%%%%%%%%%%%%%%%

In this section we recall some well-known results of the  Littlewood-Paley theory and Bony's paraproduct. These results will be  fundamental tools in the  proof of our Carleman estimate. 

%%%%%%%%%%%%%%%%%%%%%%
\subsection{Littlewood-Paley theory}%%%%%
%%%%%%%%%%%%%%%%%%%%%%%

Let $\chi$ and $\varphi$ be two functions in $C^\infty_0(\R^n_\xi)$, with values in $[0,1]$, such that \begin{eqnarray}
\nonumber && \supp(\varphi) \subseteq \{\xi \in \R_\xi^n : \frac{3}{4} \leq |\xi| \leq \frac{8}{3} \} ,\\
\label{ProjBall} && \supp(\chi) \subseteq \{\xi \in \R_\xi^n : |\xi| \leq \frac{4}{3}\}.
\end{eqnarray} 
Let, for all $\xi \in \R^n_\xi$,
 \begin{equation*}
\chi(\xi) + \sum\limits_{q\geq 0} \varphi(2^{-q}\xi) = 1,
\end{equation*} 
i.e. $\varphi(\xi) = \chi(\frac{\xi}{2})-\chi(\xi)$. By these choices we have \begin{equation*}
\supp(\chi(2^{-q}\cdot)) \subseteq \{\xi \in \R^n_\xi : |\xi| \leq \frac{4}{3} 2^q \}
\end{equation*} and therefore \begin{equation*}
\supp(\varphi(2^{-q}\cdot)) \subseteq \{\xi \in \R^n_\xi : \frac{3}{4}2^{q} \leq |\xi| \leq \frac{8}{3} 2^q \}.
\end{equation*} 
We get 
\begin{equation} \label{lpfond2}
\supp(\varphi(2^{-q}\cdot)) \cap \supp(\varphi(2^{-p}\cdot)) = \emptyset, \quad {\text {for all}}\quad  |p-q| \geq 2.
\end{equation} 
With this preparations, we define the Littlewood-Paley decomposition. Let us denote
by $\mathcal F$ the Fourier transform on $\R^n$ and by $\mathcal F^{-1}$ its inverse. Let $\Delta_q$ and $S_q$, for $q \in \Z$, be defined as follows: \begin{eqnarray*} \label{defnotationdyadique}
&& \Delta_q u := 0 \quad \text{if} \,\, q \leq -2, \\[0.2 cm]
&& \Delta_{-1} u := \chi(D_x)u = \mathcal F^{-1}(\chi(\cdot) \mathcal F(u)(\cdot)), \\[0.2 cm]
&& \Delta_q u := \varphi(2^{-q}D_x)u = \mathcal F^{-1}(\varphi(2^{-q}\cdot)\mathcal F(u)(\cdot)), \quad q \geq 0
\end{eqnarray*} 
and \begin
{equation*}
S_q u = \chi(2^{-q}D_x)u = \mathcal F^{-1}(\chi(2^{-q}\cdot)\mathcal F(u)(\cdot)) = \sum\limits_{p \leq q-1} \Delta_p u, \quad q \geq 0.
\end{equation*} 
Furthermore we denote 
\begin{equation*}
\spec (u) := \supp(\mathcal{F}(u)).
\end{equation*} 
For $u \in \mathcal S'(\R^n_x)$, 
\begin{equation*}
u = \sum\limits_{q \geq -1} \Delta_q u
\end{equation*} in the sense of $\mathcal S'(\R^n_x)$.

The following two propositions describe the decomposition and synthesis of the classical Sobolev spaces $H^s$, via Littlewood-Paley decomposition. A proof of these two propositions can be found in \cite[Prop. 4.1.11 and Prop. 4.1.12]{Metivier}.

\begin{prop} \label{SobolevLP} Let $s \in \R$. Then a tempered distribution $u \in \mathcal S'(\R^n_x)$ belongs to $H^s(\R^n_x)$ if and only if the following two conditions hold: \begin{enumerate}[(i)]
\item for all $q \geq -1$, $\Delta_q u \in L^2(\R^n_x)$,
\item the sequence $(\delta_q)_{q \in \Z_{\geq -1}}$, where $\delta_q := 2^{qs}\|\Delta_q u\|_{L^2(\R^n_x)}$, belongs to $l^2(\Z_{\geq -1})$.
\end{enumerate} Moreover, there exists $C_s \geq 1$ such that, for all $u \in H^s(\R^n_x)$, we have \begin{equation*}
\frac{1}{C_s}\|u\|_{H^s(\R^n_x)} \leq \| (\delta_q )\|_{l^2(\Z_{\geq -1})} \leq C_s \|u\|_{H^s(\R^n_x)}.
\end{equation*}
\end{prop}

\begin{prop} \label{RevSobolevLP}  Let $s \in \R$ and $R \in \R_{>1}$. Suppose that  a sequence $(u_q)_{q \in \Z_{\geq -1}}$ in $L^2(\R^n_x)$ satisfies \begin{enumerate}[(i)]
\item $\spec(u_{-1}) \subseteq \{\xi \in \R_\xi^n : |\xi| \leq R \}$ and, for all $q \geq 0$, \begin{equation*}
\spec(u_q) \subseteq \{\xi \in \R_\xi^n : R^{-1} 2^{q} \leq |\xi| \leq 2R 2^{q}\},
\end{equation*} 
\item the sequence $(\delta_q)_{q \geq -1}$, where $\delta_q := 2^{qs}\|u_q\|_{L^2(\R^n_x)}$, belongs to $l^2(\Z_{\geq -1})$.
\end{enumerate} 
Then $u = \sum\limits_{q \geq -1} u_q \in H^s(\R^n_x)$ and there exists $C_s \geq 1$ such that, for all $u \in H^s(\R^n_x)$, we have \begin{equation*}
\frac{1}{C_s}\|u\|_{H^s(\R^n_x)} \leq \| \delta_q \|_{l^2(\Z_{\geq -1})} \leq C_s \|u\|_{H^s(\R^n_x)}.
\end{equation*} When $s > 0$ it is enough to assume, instead if (i),  that, for all $q \geq -1$,
 \begin{equation*}
\spec(u_q) \subseteq \{\xi \in \R_\xi^n : |\xi| \leq R 2^q \}.
\end{equation*}
\end{prop}

The following result will be crucial in the sequel.

\begin{prop} \label{BernsteinComm} There exists a constant $C>0$ such that the following estimates hold true: \begin{enumerate}[(i)]
\item (Bernstein inequalities) for $u \in L^p(\R^n_x)$, $p \in [1+\infty]$: \begin{eqnarray*}
&& \|\nabla_x S_q u\|_{L^p(\R^n_x)}\le C 2^{q} \|u\|_{L^p(\R^n_x)}, \quad q \geq 0, \\
&&  \frac{1}{C}\|\Delta_q u\|_{L^p(\R^n_x)} \leq 2^{-q}\|\nabla_x \Delta_q u \|_{L^p(\R^n_x)} \leq C \|\Delta_q u \|_{L^p(\R^n_x)}, \quad q \geq 0.
\end{eqnarray*} For $q=-1$ only $\|\nabla_x \Delta_{-1} u \|_{L^p(\R^n_x)} \leq C \|\Delta_{-1} u \|_{L^p(\R^n_x)}$ holds.
\item (Commutator estimate) for $a \in L^\infty(\R^n_x)$ and $u \in L^2(\R^n_x)$: \begin{equation} \label{CM}
\| [S_{q'}a,\Delta_q] \Delta_p u \|_{L^2(\R^n_x)}\le C 2^{-p} \| \nabla  S_{q'}a \|_{L^\infty(\R^n_x)} \|\Delta_p u \|_{L^2(\R^n_x)}, \quad q'\geq 0, \,\, p,q \geq -1.
\end{equation}
\end{enumerate} \end{prop}

\begin{proof} The proof of the Bernstein inequalities can be found in \cite[Cor. 4.1.17]{Metivier}. The commutator estimate follows from  \cite[Th. 35]{CM1978}. This result applied to our case reads
\begin{equation} \label{CM1}
\|[a,\Delta_q]\partial_{x_k}u\|_{L^2(\R_x^n)} \leq C \|\nabla_x a\|_{L^\infty(\R_x^n)} \|u\|_{L^2(\R_x^n)}
\end{equation} for $a \in \Lip(\R_x^n)$ and $u \in H^1(\R_x^n)$. Estimate \eqref{CM} follows from \eqref{CM1} writing $\Delta_q u$ as a sum of derivatives. \end{proof}

The proof of the following proposition can be found in \cite[Prop. 1.5]{Taylor}.

\begin{prop} 
\label{PropMLP} Let $\omega$ be a modulus of continuity. 
Then, for all $u\in C^\omega(\R^n_x)$,  
 \begin{equation}
 \|\nabla_x S_qu\|_{L^\infty(\R^n_x)} \leq C 2^q \omega(2^{-q}).
 \label{nablaS}
\end{equation} 
Conversely, given $u \in L^\infty(\R^n_x)$, if \eqref{nablaS} holds, then $u \in C^\sigma(\R^n_x)$, where 
$\sigma(h) = \int\limits_0^h \frac{\omega(t)}{t} dt.$
\end{prop}

The main consequence of Proposition \ref{PropMLP} is contained in the following corollary.

\begin{cor} \label{Cor1}
Let $\omega$ be a modulus of continuity satisfying condition \eqref{Dini1}.
Then a function $u \in L^\infty(\R^n_x)$ belongs to $C^\omega(\R^n_x)$ if and only if 
\begin{equation}\label{charComega}
\sup\limits_{q \in \N_0} \frac{\|\nabla_x (S_q u)\|_{L^\infty(\R^n_x)}}{2^q\omega(2^{-q})}  < +\infty.
\end{equation}
\end{cor}

Other interesting properties of the Littlewood-Paley decomposition are contained in the following proposition.

\begin{prop} \label{ComEst1} 
Let $a \in C^\omega(\R^n_x)$. Then, for all $q\geq -1$
\begin{equation} \label{Deltaa}
\|\Delta_q a\|_{L^\infty(\R^n_x)} \leq C \|a\|_{C^\omega(\R^n_x)}\omega(2^{-q}),
\end{equation} 
and, if additionally \eqref{Dini1} holds, 
\begin{equation} \label{pos}
\|a-S_qa\|_{L^\infty(\R^n_x)} \leq C\|a\|_{C^\omega(\R_x^n)} \omega(2^{-q}).
\end{equation} \end{prop}

\begin{proof} The proof of \eqref{Deltaa}  is the same as \cite[Prop. 3.4]{CL}. To prove the second estimate we note that 
\begin{equation*}
a-S_qa = \sum\limits_{p \geq q} \Delta_p a
\end{equation*} and therefore, from \eqref{Deltaa}, we get 
\begin{equation*}
\|a-S_qa\|_{L^\infty(\R^n_x)} \leq \sum\limits_{p \geq q} \|\Delta_p a\|_{L^\infty(\R_x^n)} \leq C \|a\|_{C^\omega(\R_x^n)} \sum\limits_{p \geq q}\omega(2^{-p}).
\end{equation*} 
An elementary computation gives that \eqref{Dini1} is equivalent to 
$\sum\limits_{p \geq q} \omega(2^{-p}) \leq \omega(2^{-q+1})$. This concludes the proof.
\end{proof}

\begin{rem}\label{Pos} 
Estimate \eqref{pos} implies that $(S_qa_{jk})_{j,k=1}^n$ is a positive matrix if $(a_{jk})_{j,k=1}^n$ is a positive matrix and $q$ is sufficiently large. \end{rem}

For later use we introduce a weighted Sobolev space.

\begin{defi} \label{WeightSobolev} Let $s \in \R$ and $\omega$ be a modulus of continuity. Let  $\Omega(q)=2^{q} \omega(2^{-q})$.  We say that $u \in \mathcal S'(\R_x^n)$ belongs to $H^{s}_{\Omega}(\R_x^n)$ if 
\begin{equation*}
\|u\|_{H^{s}_{\Omega}(\R_x^n)}:=\Big (\sum\limits_{q \geq -1} 2^{2sq}\Omega^2(q) \|\Delta_q u\|_{L^2(\R_x^n)}^2\Big)^{1\over 2} < +\infty.
\end{equation*}
\end{defi}

%%%%%%%%%%%%%%%%%%%%
\subsection{Bony's paraproduct}%%%%
%%%%%%%%%%%%%%%%%%%%

Let us now define Bony's paraproduct (see \cite{Bony81}) for tempered distributions $u$ and $v$ as 
\begin{equation*}
T_u v= \sum\limits_{q \geq 1} \sum_{p \leq q-2} \Delta_p u \Delta_q v = \sum\limits_{q\geq 1} S_{q-1}u\Delta_q v.
\end{equation*} 
Let us define also
 \begin{equation*}
R(u,v)=\sum_{\substack{q \geq -1 \\ i\in\{0,\pm 1\}} }\Delta_{q} u \Delta_{q+i} v = \sum\limits_{q \geq -1} \Delta_q u \tilde{\Delta}_q v, \quad \tilde{\Delta}_q := \Delta_{q-1} + \Delta_q + \Delta_{q+1}.
\end{equation*} With this we can (formally) decompose a product $uv$ with $u$, $v$ $\in \mathcal S'(\R^n_x)$ by \begin{equation*}
uv = T_u v + T_v u + R(u,v).
\end{equation*} 

\begin{prop} Let $a \in L^\infty(\R^n_x)$, $s \in \R$. Then the operator $T_a$ maps $H^s(\R_x^n)$ continuously into $H^s(\R_x^n)$, i.e. there exist a constant $C_s > 0$ such that 
\begin{equation*}
\|T_a u\|_{H^s(\R_x^n)} \leq C_s \|a\|_{L^ \infty(\R^n_x)} \|u\|_{H^s(\R_x^n)}.
\end{equation*} \end{prop}

The proof of this proposition can be found in \cite[Prop. 5.2.1]{Metivier}. Other mapping properties, especially of the remainder $R(u,v)$, will be proved in Section \ref{remainderest}. \\

Let now $a$ and $b$ be tempered distributions sufficiently regular such that $ab$ makes sense. Then we have \begin{equation*}
\Delta_q(ab)=\Delta_q T_a b+\Delta_q T_b a+\Delta_q R(a,b)=\Delta_q T_a b+\Delta_q \tilde R(a,b),
\end{equation*} where \begin{equation} \label{eq:Rem}
\tilde R(a,b)=T_b a +R(a,b)=\sum_{q' \geq -1} S_{q'+2} b \Delta_{q'} a.
\end{equation}
From the definition of $\Delta_q$ and $S_q$ it is easy to verify that 
\begin{equation}\label{supports}
\Delta_q(S_{q'-1}a\Delta_{q'}b)=0 \quad {\text{ if}} \quad |q'-q|\geq 5,
\end{equation}
and similarly 
\begin{equation}\label{supports2}
\Delta_q(S_{q'+2}a\Delta_{q'}b)=0 \quad{\text{ if}},\quad q'\leq q-4,
\end{equation}
so that

\begin{eqnarray*} 
 \Delta_q (ab) &=&\sum\limits_{|q'-q|\le 4}\Delta_q (S_{q'-1}a\Delta_{q'}b)+\sum\limits_{q'> q-4}\Delta_q(S_{q'+2}b\Delta_{q'}a)\\
&=&\sum\limits_{|q'-q|\le 4}[\Delta_q,S_{q'-1}a]\Delta_{q'}b+\sum\limits_{|q'-q|\le 4}S_{q'-1}a\Delta_q\Delta_{q'}b \\
&& \qquad +\sum\limits_{q'>q-4}\Delta_q(S_{q'+2}b\Delta_{q'}a)\\
&=&\sum\limits_{|q'-q|\le 4}[\Delta_{q},S_{q'-1} a]\Delta_{q'}b+\sum\limits_{|q'-q|\le 4}(S_{q'-1} a-S_{q-1}a)\Delta_q\Delta_{q'}b \\
&& \qquad +\sum\limits_{q'> q-4}\Delta_q (S_{q'+2} b\Delta_{q'}a)+\underbrace{\sum\limits_{|q'-q|\le 4} S_{q-1}a \Delta_q\Delta_{q'} b }_{= S_{q-1}a \Delta_q b}.
\end{eqnarray*} 
Consequently, \begin{equation}\label{Decomposition}
  \Delta_q (ab) = S_{q-1}a \Delta_q b + \mathcal R_q(a,b),
\end{equation} 
where
 \begin{eqnarray*}
&& \mathcal R_q(a,b) = \sum\limits_{|q'-q|\le 4}[\Delta_{q},S_{q'-1} a]\Delta_{q'}b+\sum\limits_{|q'-q|\le 4}(S_{q'-1} a-S_{q-1}a)\Delta_q\Delta_{q'}b \\
&& \qquad \qquad \qquad + \sum\limits_{q'> q-4}\Delta_q (S_{q'+2} b\Delta_{q'}a) = \mathcal R_q^{(1)}(a,b) + \mathcal R_q^{(2)}(a,b) + \mathcal R_q^{(3)}(a,b).
\end{eqnarray*}
Let us remark that a consequence of \eqref{Decomposition} is that
\begin{equation}
\label{specRq}
\spec R_q(a,b)\subseteq \{\xi\in\R^n_\xi\,:\, |\xi|\leq {10 \over 3}2^q\}.
\end{equation}

%%%%%%%%%%%%%%%%%%%%%%%%%%%%%%%%%%%%%%%%%
\subsection{Auxiliary estimates for $\mathcal R_q(a,b)$} \label{remainderest}%%%%
%%%%%%%%%%%%%%%%%%%%%%%%%%%%%%%%%%%%%%%%%

In this section we prove an estimate about $\mathcal R_q(a,b)$ which we will use in the sequel.

\begin{lem}\label{lemma}
Let $\omega$ be a modulus of continuity satisfying \eqref{Dini1},  $s\in \R$, $\Omega(q)$ as in \hbox{Definition \ref{WeightSobolev}}. Let $a \in C^\omega(\R_x^n)$ and $b \in H_\Omega^{-s}(\R_x^n)$. Then 
 \begin{equation}\label{estimateR}
\Big( \sum_{q \geq -1} 2^{2(1-s)q} \left\|\mathcal R_q^{(i)}(a,b)\right\|^2_{L^2(\R_x^n)}\Big)^{\frac{1}{2}} \leq C_{s,i }\|a\|_{C^\omega(\R_x^n)}\|b\|_{H^{-s}_{\Omega}(\R_x^n)}
, \quad i=1,\ 2.
\end{equation} 
Suppose moreover that $s\in (0,1)$ and $\omega$ satisfies \eqref{TechCond1} and \eqref{TechCond2}.
Then the estimate \eqref{estimateR} holds also for $i=3$.
\end{lem}

\begin{proof}
Let us start with the inequality \eqref{estimateR}, for $i=1$. We have
\begin{eqnarray}\label{sumR}
\mathcal R_q^{(1)}(a,b) &=& \sum\limits_{|q'-q|\leq 4} [\Delta_q, S_{q'-1}a]\Delta_{q'}b \nonumber \\
&=& [\Delta_q, S_{q-5}a]\Delta_{q-4}b+[\Delta_q, S_{q-4}a]\Delta_{q-3}b+\dots+[\Delta_q, S_{q+3}a]\Delta_{q+4}b. 
\end{eqnarray}
Consider the first term of this sum. We have, from \eqref{CM} and \eqref{charComega},
\begin{eqnarray*}
\| [\Delta_q, S_{q-5}a]\Delta_{q-4}b\|_{L^2(\R_x^n)} &\leq& C2^{-(q-4)}\|\nabla_x S_{q-5} a\|_{L^\infty(\R_x^n)} \|\Delta_{q-4} b\|_{L^2(\R_x^n)}   \\
&\leq& {C\over 2}\omega(2^{-(q-5)})\|a\|_{C^\omega(\R^n_x)} \|\Delta_q b \|_{L^2(\R_x^n)}.
\end{eqnarray*} 
Since $u\in H_\Omega^{-s}(\R^n_x)$ we have that
$$
\|\Delta_{q-4} u\|_{L^2(\R^n_x)}\leq {2^{s(q-4)}\over \Omega(q-4)}\eps_q = {2^{(s-1)(q-4)}\over \omega(2^{-(q-4)})}\eps_q,
$$
where $(\eps_q)_{q \in \Z_{\geq -1}}$ is a sequence in $l^2(\Z_{\geq -1})$ and there exists $c_s\geq 1$ such that
\begin{equation}\label{equivnorm}
{1\over c_s}\, \|b\|_{H_\Omega^{-s}(\R^n_x)}\leq \|(\eps_q)\|_{l^2(\Z_{\geq -1})}\leq c_s\, \|b\|_{H_\Omega^{-s}(\R^n_x)}.
\end{equation}
We get 
\begin{equation}\label{estimateR'}
2^{(1-s)q}\| [\Delta_q, S_{q-5}a]\Delta_{q-4}b\|_{L^2(\R_x^n)} \leq C 2^{3-4s}{\omega(2^{-(q-5)})\over \omega(2^{-(q-4)})}\|a\|_{C^\omega(\R_x^n)}  \eps_{q-4} \leq C_s\|a\|_{C^\omega(\R_x^n)} \eps_{q-4}.
\end{equation} 
For all the other terms in \eqref{sumR} we obtain  an  estimate similar to \eqref{estimateR'} and the inequality \eqref{estimateR} follows. 

Let us now consider the inequality \eqref{estimateR}, for $i=2$.
We have \begin{equation*}
\|\mathcal R_q^{(2)}(a,b)\|_{L^2(\R_x^n)} = \left\|(S_{q-2}-S_{q-1})a \Delta_q\Delta_{q-1} b + (S_{q}-S_{q-1})a \Delta_q\Delta_{q+1}b \right\|_{L^2(\R_x^n)}.
\end{equation*} 
Since $S_{q-2}-S_{q-1} = -\Delta_{q-2}$ and $S_{q}-S_{q-1}= \Delta_{q-1}$,  we deduce from \eqref{Deltaa}, 
\begin{eqnarray*}
\|\mathcal R_q^{(2)}(a,b)\|_{L^2(\R_x^n)} &\leq&  (\|\Delta_{q-2}a\|_{L^\infty(\R_x^n)}+\|\Delta_{q-1}a\|_{L^\infty(\R_x^n)})\|\Delta_q b\|_{L^2(\R_x^n)} \\
&\leq& 2\,C \|a\|_{C^\omega(\R_x^n)} \omega(2^{-q}) \frac{2^{sq}}{\Omega(q)}\eps_q,
\end{eqnarray*} 
where we have used the fact that  $\|\Delta_q b\|_{L^2(\R^n_x)} \leq \frac{2^{qs}}{\Omega(q)} \eps_q$, where $(\eps_q)_{q \in \Z_{\geq -1}} $ is a sequence in $ l^2(\Z_{\geq -1})$
satisfying \eqref{equivnorm}.
Therefore, remembering that $\Omega(q)=2^q\omega(2^{-q})$, we get \begin{eqnarray*}
2^{(1-s)q}\|\mathcal R_q^{(2)}(a,b)\|_{L^2(\R_x^n)} \leq 2\,C\|a\|_{C^\omega(\R_x^n)} \eps_q.
\end{eqnarray*} 
Thus, inequality  \eqref{estimateR}, for $i=2$, follows. 
Let now $s\in (0,1)$. We have \begin{eqnarray*}
R_q^{(3)}(a,b) &=& \sum\limits_{q' > q-4} \Delta_q(S_{q'+2}b \Delta_{q'}a) \\
&=& \sum\limits_{q' > q-4} \Big( \Delta_q(S_{q'-1}b \Delta_{q'}a) +\Delta_q\big(\Delta_{q'-1}b \Delta_{q'}a + \Delta_{q'}b \Delta_{q'}a +\Delta_{q'+1}b \Delta_{q'}a\big) \Big).
\end{eqnarray*}
From \eqref{supports}  and \eqref{supports2} we obtain 
\begin{equation} 
\begin{array}{ll}
\displaystyle{R_q^{(3)}(a,b)}= \displaystyle{ \Delta_q\big( S_{q-4}b\Delta_{q-3} a + \dots + S_{q+4}b\Delta_{q+5}a  \big)} \\[0.3cm]
 \displaystyle{\qquad\qquad\qquad \qquad+ \sum\limits_{q'\geq -1} \Big( \Delta_q\big(\Delta_{q'-1}b \Delta_{q'}a + \Delta_{q'}b \Delta_{q'}a +\Delta_{q'+1}b \Delta_{q'}a\big) \Big).}
\end{array} \label{Calc1}
\end{equation}
The nine terms in the first line in \eqref{Calc1} are essentially of the form $\Delta_q (S_{q-1}b\Delta_q a)$ and can be treated as follows: 
\begin{eqnarray*}
\sum\limits_{q \geq -1} 2^{2(1-s)q}\|\Delta_q (S_{q-1}b\Delta_q a)  \|_{L^2(\R^n_x)}^2 &\leq& \sum\limits_{q \geq -1} 2^{2(1-s)q}\|S_{q-1}b\Delta_q a \|_{L^2(\R^n_x)}^2 \\
&\leq& \sum\limits_{q \geq -1} 2^{2(1-s)q} \|S_{q-1}b \|_{L^2(\R^n_x)}^2 \|\Delta_q a \|_{L^\infty(\R^n_x)}^2 \\
&\leq& \sum\limits_{q \geq -1} 2^{2(1-s)q} \big( \sum\limits_{p \leq q-2} \|\Delta_p b\|_{L^2(\R^n_x)} \big)^2 \|\Delta_q a \|_{L^\infty(\R^n_x)}^2 \\
&\leq& \sum\limits_{q \geq -1} 2^{2(1-s)q} \Big( \sum\limits_{p \leq q-2} \frac{2^{ps}}{\Omega(p)} \eps_p\Big)^2 2^{-2q} \Omega^2(q) \|a\|_{C^\omega(\R^n_x)}^2 \\
&\leq& \sum\limits_{q \geq -1} \Big( \sum\limits_{p \leq q-2} 2^{-s(q-p)}\frac{\Omega(q)}{\Omega(p)} \eps_p \Big)^2 \|a\|_{C^\omega(\R^n_x)}^2,
\end{eqnarray*} 
where $(\eps_j)_{j \in \Z_{\geq -1}}$ is a sequence in $l^2(\Z_{\geq -1})$ with \eqref{equivnorm}.
From \eqref{TechCond1} and the definition of $\Omega$ we get
\begin{equation*}
\tilde{\eps}_q := \sum\limits_{p \leq q-2} 2^{-s(q-p)}\frac{\Omega(q)}{\Omega(p)} \eps_p\leq \sum\limits_{p \leq q-2} 2^{(1-s)(q-p)}\omega(2^{-(q-p)})\eps_p.
\end{equation*}
Then \eqref{TechCond2} and the Young inequality for convolution in $l^p$ spaces give that the sequence $(\tilde \eps_j)_{j \in \Z_{\geq -1}}$ is in $l^2(\Z_{\geq -1})$ 
and there exists $C_s>0$ such that \begin{equation*}
\|(\tilde \eps_j)\|_{l^2(\Z_{\geq -1})}\leq \tilde C_s \|(\eps_j)\|_{l^2(\Z_{\geq -1})}.
\end{equation*}
From \eqref{equivnorm} we conclude that \begin{equation*}
\sum\limits_{q \geq -1} 2^{2(1-s)q}\|\Delta_q (S_{q-1}b\Delta_q a)  \|_{L^2(\R_x^n)}^2 \leq
\tilde C^2_s \|(\eps_j)\|^2_{l^2(\Z_{\geq -1})} \|a\|_{C^\omega(\R^n_x)}^2 \leq C^2_s \|b\|^2_{H^{-s}_{\Omega}(\R_x^n)} \|a\|_{C^\omega(\R_x^n)}^2.
\end{equation*}
The second line of \eqref{Calc1} is a sum of three terms of the form $\sum_{q'\geq -1}\Delta_q\big(\Delta_{q'} b\Delta_{q'} a\big)$. 
We have \begin{equation*}
\sum_{q\geq -1} 2^{2(1-s)q}\|\sum_{q'\geq -1}\Delta_q\big(\Delta_{q'} b\Delta_{q' }a\big)\|^2_{L^2(\R_x^n)}= 
\sum_{q\geq -1} 2^{2(1-s)q}\|\Delta_q\big(\sum_{q'\geq -1}\Delta_{q' }b\Delta_{q'} a\big)\|^2_{L^2(\R_x^n)}.
\end{equation*}
Thanks to the result of Proposition \ref{SobolevLP}, this last quantity coincides with $\|\sum_{q'\geq -1}\Delta_{q' }b\Delta_{q'} a\|^2_{H^{1-s}(\R_x^n)}$. To compute the 
$H^{1-s}(\R_x^n)$ of $\sum_{q'\geq -1}\Delta_{q' }b\Delta_{q'} a$ we use Proposition \ref{RevSobolevLP}. In fact $1-s>0$, \begin{equation*}
\spec (\Delta_{q' }b\Delta_{q'} a)\subseteq
\{\xi\in\R^n_\xi\,:\, |\xi|\leq {16\over 3}2^{q'}\},
\end{equation*} and \begin{equation*}
2^{(1-s)q'}\|\Delta_{q' }b\Delta_{q'} a\|_{L^2(\R_x^n)}\leq 2^{(1-s)q'}\|\Delta_{q' }b\|_{L^2(\R_x^n)}\|\Delta_{q'} a\|_{L^\infty(\R_x^n)}
\leq \eps_{q'} \|a\|_{C^\omega(\R_x^n)}.
\end{equation*}
Again \eqref{equivnorm} gives $\|\sum_{q'\geq -1}\Delta_{q' }b\Delta_{q'} a\|^2_{H^{1-s}(\R_x^n)}\leq C_s 
 \|b\|^2_{H^{-s}_{\Omega}(\R_x^n)} \|a\|_{C^\omega(\R_x^n)}^2.$
The proof of the lemma is concluded.
\end{proof}

%%%%%%%%%%%%%%%%%%%%%%%%%%%%%
\section{The Carleman estimate} \label{SecCarleman}%%%
%%%%%%%%%%%%%%%%%%%%%%%%%%%%%

%%%%%%%%%%%%%%%%%%%
\subsection{The weight function}%%%%
%%%%%%%%%%%%%%%%%%%

The idea of constructing a weight function which is linked to the modulus of continuity is due to Tarama (\cite{Tarama}, see also \cite{DSP, DS2012, DSP2012}).
Let $\mu$ be a modulus of continuity satisfying \eqref{Osgood}. We set
\begin{equation*}
\varphi(t) := \int\limits_\frac{1}{t}^1 \frac{1}{\mu(s)} ds.
\end{equation*} 
The function $\varphi$ is strictly increasing and $C^1([1,+\infty[)$. We have $\varphi([1,+\infty)) = [0,+\infty)$ and $\varphi'(t)= 1/(t^2\mu(1/t))>0$ for all $t \in [1,+\infty)$. 
We define 
\begin{equation}\label{Phi}
\Phi(\tau) := \int\limits_0^\tau \varphi^{-1}(s) ds.
\end{equation}
From this we get $\Phi'(t) = \varphi^{-1}(t)$ 
and therefore 
$\lim_{\tau \rightarrow +\infty} \Phi'(\tau) = +\infty$. 
Moreover we have 
\begin{equation} \label{ODE}
\Phi''(\tau) = \left(\Phi'(\tau)\right)^2 \mu\left( \frac{1}{\Phi'(\tau)} \right)
\end{equation} 
for all $\tau \in [0,+\infty)$ 
and, since the function $\sigma \mapsto \sigma \mu(1/\sigma)$ is increasing on the interval $[1,+\infty)$, we obtain that 
\begin{equation*}
\lim\limits_{\tau \rightarrow +\infty} \Phi''(\tau) = \lim\limits_{\tau \rightarrow +\infty} \left(\Phi'(\tau)\right)^2 \mu\left( \frac{1}{\Phi'(\tau)} \right) = +\infty.
\end{equation*}

%%%%%%%%%%%%%%%%%%%%%%%
\subsection{The Carleman estimate}%%
%%%%%%%%%%%%%%%%%%%%%%%

The uniqueness  result of Theorem \ref{Uniqueness} will be a consequence of  the following Carleman estimate.
\begin{prop} \label{Carleman} 
Let $\mu$ and $\omega$ be two moduli of continuity such that $\omega(s)= \sqrt{\mu(s^2)}$.
Suppose that $\mu$ and $\omega$ satisfy \eqref{Osgood} and \eqref{Dini1}, \eqref{TechCond1}, \eqref{TechCond2} respectively.
Suppose that, for all $j,k = 1, \dots, n$,
\begin{equation*}
a_{jk} \in C^\mu([0,T],L^\infty(\R_x^n)) \cap L^\infty([0,T],C^\omega(\R_x^n)),
\end{equation*}
 and let \eqref{elliptic} hold. Let $\Phi$   and $H_\Omega^s (\R^n_x)$ defined in \eqref{Phi} and Definition \ref {WeightSobolev} respectively. Let $s\in (0,1)$.
Then there exist $\gamma_0 \geq 1$, $C>0$, such that, for all $\gamma \geq \gamma_0$ and all $u \in C^\infty_0(\R_t \times \R^n_x)$ with $\supp(u) \subseteq [0,T/2] \times \R^n_x$, \begin{eqnarray}
\nonumber && \int\limits_0^{T/2}e^{\frac 2\gamma\Phi(\gamma(T-t))}\big\|\partial_t u+\sum\limits_{j,k=1}^n \partial_{x_j}(a_{jk}(t,x)\partial_{x_k} u)\big\|_{H^{-s}(\R^n_x)}^2 dt \geq \\
\label{eq:CarlemanDSJP} &&  \qquad\qquad\qquad \qquad C\gamma^{1/4}\int\limits_0^{T/2}e^{\frac 2\gamma\Phi(\gamma (T-t))}\big(\|\nabla u\|^2_{H^{-s}_{\Omega}(\R^n_x)}+\gamma^{3/4}\|u\|^2_{L^2(\R^n_x)}\big) dt.
\end{eqnarray}
\end{prop}

Setting 
\begin{equation*}
u(t,x) = e^{-\frac{1}{\gamma}\Phi(\gamma(T-t))}v(t,x),
\end{equation*} 
the Carleman estimate \eqref{eq:CarlemanDSJP} becomes 
\begin{eqnarray}
\nonumber && \int\limits_0^{T/2}\big\|\partial_t v + \sum\limits_{j,k=1}^n \partial_{x_j}(a_{jk}(t,x)\partial_{x_k} v) + \Phi'(\gamma(T-t))v \big\|_{H^{-s}(\R^n_x)}^2 dt \geq \\
\label{eq:CarlemanTrans} &&  \qquad\qquad\qquad \qquad \qquad \qquad \qquad C\gamma^{1/4}\int\limits_0^{T/2} \left( \|\nabla v\|^2_{H^{-s}_{\Omega}(\R^n_x)}+\gamma^{3/4}\|v\|^2_{L^2(\R^n_x)} \right) dt.
\end{eqnarray}
The proof of such inequality is divided in several steps which we will present in the subsequent subsections. 

%%%%%%%%%%%%%%%%%%%%%%%%%%%%%%
\subsection{Regularization in $t$} \label{regularization}%%%
%%%%%%%%%%%%%%%%%%%%%%%%%%%%%%

In our proof of the Carleman estimate we need to perform some integrations by part with respect to $t$ and if the coefficients $a_{jk}$ are not sufficiently regular this is not possible. We will avoid this difficulty regularizing  the $a_{jk}$'s  with respect to $t$ and to this end we will use Friedrichs mollifiers. We take a $\rho \in C_0^\infty(\R)$ with $\supp(\rho) \subseteq [-\frac{1}{2},\frac{1}{2}]$ and $\int\limits_\R \rho(\tau) d\tau = 1$ and $\rho(\tau) = \rho(-\tau)$ and we define 
\begin{equation*}
a_{jk}^\eps(t,x) := \frac{1}{\eps} \int\limits_{\R^n} a(s,x) \rho\left( \frac{t-s}{\eps} \right) ds.
\end{equation*} 
 We have easily
 \begin{equation*}
|a_{jk}^\eps(t,x) - a_{jk}(t,x)| \leq C \mu(\eps)
\end{equation*} and \begin{equation*}
|\partial_t a_{jk}^\eps(t,x)| \leq C \frac{\mu(\eps)}{\eps}.
\end{equation*}
where $C$ depends only on $\|a_{j,k}\|_{C^\mu([0,T], L^\infty(\R^n_x))}$.
%%%%%%%%%%%%%%%%%%%%%%%%%%%%%%
\subsection{Estimates for the microlocalized operator}%%%
%%%%%%%%%%%%%%%%%%%%%%%%%%%%%%

Using the characterization of Sobolev spaces given in Proposition \ref{SobolevLP} we have that the left hand side part of \eqref{eq:CarlemanTrans} reads
\begin{equation} 
\label{Dyadicdec}
\sum\limits_{q \geq -1} 2^{-2sq} \int\limits_0^{T/2} \Big\| \partial_t v_q + \sum\limits_{j,k=1}^n \partial_{x_j}\left(\Delta_q (a_{jk}(t,x)\partial_{x_k} v)\right) + \Phi'(\gamma(T-t))v_q \Big\|^2_{L^2(\R^n_x)} dt.
\end{equation} 
where we set $\Delta_q v := v_q$. 
We use formula \eqref{Decomposition} and we replace $a_{jk}(t,x)\partial_{x_k} v$ with $(S_{q-1}a_{jk}(t,x))\partial_{x_k}v_q + \mathcal R_q(a_{jk},\partial_{x_k}v)$. We deduce that \eqref{Dyadicdec} is bounded from below by
\begin{eqnarray*}
\sum\limits_{q \geq -1} 2^{-2sq-1} \int\limits_0^{T/2} \Big\| \partial_t v_q + \sum\limits_{j,k=1}^n \partial_{x_j}\left(S_{q-1}a_{jk}(t,x)\partial_{x_k}v_q\right) + \Phi'(\gamma(T-t))v_q \Big\|^2_{L^2(\R^n_x)} dt\qquad\qquad \\
- \sum\limits_{q \geq -1} 2^{-2sq} \int\limits_0^{T/2} \Big\|\partial_{x_j}\left( \mathcal R_q(a_{jk},\partial_{x_k}v)\right)\Big\|^2_{L^2(\R^n_x)}dt.
\end{eqnarray*}
We use now \eqref{specRq}, the Bernstein inequalities and the result of Lemma \ref{lemma} and we get
\begin{equation*}
\sum\limits_{q \geq -1} 2^{-2sq} \int\limits_0^{T/2} \Big\|\partial_{x_j}\left( \mathcal R_q(a_{jk},\partial_{x_k}v)\right)\Big\|^2_{L^2(\R^n_x)}dt
\leq C \int\limits_0^{T/2}  \|\nabla v\|^2_{H^{-s}_{\Omega}(\R^n_x)}  dt,
\end{equation*}
where $C$ depends only on $s$ and on $\max_{j,k}\|a_{j,k}\|_{L^\infty([0,T],C^\omega(\R^n_x))}$.
Finally, \eqref{eq:CarlemanTrans} will be a consequence of
\begin{eqnarray*}
\sum\limits_{q \geq -1} 2^{-2sq} \int\limits_0^{T/2} \Big\| \partial_t v_q + \sum\limits_{j,k=1}^n \partial_{x_j}\left(S_{q-1}a_{jk}(t,x)\partial_{x_k}v_q)\right) + \Phi'(\gamma(T-t))v_q \Big\|^2_{L^2(\R^n_x)} dt \geq\qquad\qquad\\
C\gamma^{1/4}\int\limits_0^{T/2} \left( \|\nabla v\|^2_{H^{-s}_{\Omega}(\R^n_x)}+\gamma^{3/4}\|v\|^2_{L^2(\R^n_x)} \right) dt.
\end{eqnarray*}
We have
\begin{eqnarray*}
&&\!\!\!\!\!\!\!\! \int\limits_0^{T/2} \Big\| \partial_t v_q + \sum\limits_{j,k=1}^n \partial_{x_j}\left(S_{q-1}a_{jk}(t,x)\partial_{x_k}v_q\right) + \Phi'(\gamma(T-t))v_q \Big\|^2_{L^2(\R^n_x)} dt \\
&&= \int\limits_0^{T/2} \|\partial_t v_q\|^2_{L^2(\R_x^n)} dt +\int\limits_0^{T/2} \Big\| \sum\limits_{j,k=1}^n \partial_{x_j}\left(S_{q-1}a_{jk}(t,x)\partial_{x_k}v_q \right) + \Phi'(\gamma(T-t))v_q \Big\|^2_{L^2(\R^n_x)} dt \\
 && \quad + 2\real\int\limits_0^{T/2} \la \partial_t v_q \,|\, \Phi'(\gamma(T-t))v_q \ra_{L^2(\R_x^n)} dt + 2\real \sum\limits_{j,k=1}^n \int\limits_0^{T/2} \la \partial_t v_q\, |\,\partial_{x_j}\left(S_{q-1}a_{jk}(t,x)\partial_{x_k}v_q \right) \ra dt .
 \end{eqnarray*} 
We compute by integration by parts \begin{eqnarray*}
2\real\int\limits_0^{T/2} \la \partial_t v_q \, |\, \Phi'(\gamma(T-t))v_q \ra_{L^2(\R_x^n)} dt = \gamma\int\limits_0^{T/2} \Phi''(\gamma(T-t))\|v_q\|_{L^2(\R^n_x)}^2 dt
\end{eqnarray*} To handle the second scalar product we use the regularization from Section \ref{regularization}. In particular
\begin{equation}
\label{approx1}  |S_qa_{jk}^\eps(t,x)-S_qa_{jk}(t,x)| \leq C\mu(\eps), \quad \text{for all}\quad (t,x) \in [0,T] \times \R_x^n
\end{equation}
and 
\begin{equation}
\label{approx2}  |\partial_t S_qa_{jk}^{\eps}(t,x)| \leq C\frac{\mu(\eps)}{\eps}, \quad \text{for all}\quad (t,x) \in [0,T] \times \R_x^n,
\end{equation} 
where $C$ depends only on $\max_{j,k}\|a_{j,k}\|_{C^\mu([0,T], L^\infty(\R^n_x))}$.
Adding and subtracting $\partial_{x_j}\left(S_{q-1}a_{jk}^\eps(t,x)\partial_{x_k}v_q\right)$ we get 
\begin{eqnarray}
\nonumber && 2\real \sum\limits_{j,k=1}^n \int\limits_0^{T/2} \la \partial_t v_q \, |\, \partial_{x_j}\left(S_{q-1}a_{jk}(t,x)\partial_{x_k}v_q\right) \ra_{L^2(\R^n_x)} dt \\
\nonumber && \qquad = 2\real \sum\limits_{j,k=1}^n \int\limits_0^{T/2} \la \partial_t v_q \, |\, \partial_{x_j}\left(S_{q-1}a_{jk}^\eps(t,x)\partial_{x_k}v_q \right) \ra_{L^2(\R_x^n)} dt \\
\label{term1} && \qquad \quad + 2\real \sum\limits_{j,k=1}^n \int\limits_0^{T/2} \la \partial_t v_q \, |\, \partial_{x_j}\left(S_{q-1}(a_{jk}(t,x)-a_{jk}^\eps(t,x))\partial_{x_k}v_q\right) \ra_{L^2(\R_x^n)} dt.
\end{eqnarray} 
By integration by parts we get 
\begin{eqnarray*}
&& 2\real \sum\limits_{j,k=1}^n \int\limits_0^{T/2} \la \partial_t v_q \, |\,\partial_{x_j}\left(S_{q-1}a_{jk}^\eps(t,x)\partial_{x_k}v_q\right) \ra_{L^2(\R_x^n)} dt \\
&& \qquad \qquad = \sum\limits_{j,k=1}^n \int\limits_0^{T/2} \la \partial_{x_j}v_q \,|\, \partial_t(S_{q-1}a_{jk}^\eps(t,x))\partial_{x_k}v_q \ra_{L^2(\R_x^n)}.
\end{eqnarray*}
From \eqref{approx2} we obtain 
\begin{eqnarray*}
&& \Big| 2\real \sum\limits_{j,k=1}^n \int\limits_0^{T/2} \la \partial_t v_q \, |\, \partial_{x_j}\left(S_{q-1}a_{jk}^\eps(t,x)\partial_{x_k}v_q\right) \ra_{L^2(\R_x^n)} dt \Big| \\
&& \qquad \leq \sum\limits_{j,k=1}^n \int\limits_0^{T/2} \|\partial_{x_j}v_q\|_{L^2(\R_x^n)}\|\partial_t(S_{q-1}a_{jk}^\eps(t,x))\partial_{x_k}v_q\|_{L^2(\R_x^n)} dt \\
&& \qquad\qquad \leq C_1 \frac{\mu(\eps)}{\eps} 2^{2q} \int\limits_0^{T/2} \|v_q\|_{L^2(\R_x^n)}^2 dt,
\end{eqnarray*} 
where we have used the fact that $\|\partial_{x_k}v_q\|_{L^2(\R_x^n)} \leq C 2^q \|v_q\|_{L^2(\R_x^n)}$ 
and 
\begin{eqnarray*}
\|\partial_t(S_{q-1}a_{jk}^\eps(t,x))\partial_{x_k}v_q\|_{L^2(\R_x^n)} &\leq& \|\partial_t(S_{q-1}a_{jk}^\eps(t,x))\|_{L^\infty(\R_x^n)} \|\partial_{x_k}v_q\|_{L^2(\R_x^n)} \\
&\leq& C 2^q \frac{\mu(\eps)}{\eps} \|v_q\|_{L^2(\R_x^n)}.
\end{eqnarray*} 
Remark that $C_1$ depends only on $\max_{j,k} \|a_{j,k}\|_{C^\mu([0,T], L^\infty(\R^n_x))}$.
For the second term in \eqref{term1} we perform one integration by parts in $x$ and the we use the Cauchy-Schwarz inequality. We get 
\begin{eqnarray*}
&& 2\real \sum\limits_{j,k=1}^n \int\limits_0^{T/2} \la \partial_t v_q \, |\, \partial_{x_j}\left(S_{q-1}(a_{jk}(t,x)-a_{jk}^\eps(t,x))\partial_{x_k}v_q\right) \ra_{L^2(\R_x^n)} dt \\
&& \qquad = -2\real \sum\limits_{j,k=1}^n \int\limits_0^{T/2} \la \partial_{x_j}\partial_t v_q \,|\, \left(S_{q-1}(a_{jk}(t,x)-a_{jk}^\eps(t,x)\right) \partial_{x_k}v_q \ra_{L^2(\R_x^n)} dt,
\end{eqnarray*} 
and then
\begin{eqnarray*}
&& \Big| 2\real \sum\limits_{j,k=1}^n \int\limits_0^{T/2} \la \partial_{x_j}\partial_t v_q \,|\, \left(S_{q-1}(a_{jk}(t,x)-a_{jk}^\eps(t,x)\right) \partial_{x_k}v_q \ra_{L^2(\R_x^n)} dt \Big| \\
&& \qquad \leq \sum\limits_{j,k=1}^n \int\limits_0^{T/2} \|\partial_{x_k}\partial_t v_q\|_{L^2(\R_x^n)} \|\left(S_{q-1}(a_{jk}(t,x)-a_{jk}^\eps(t,x)\right) \partial_{x_k}v_q\|_{L^2(\R_x^n)} dt \\
&& \qquad \leq C \sum\limits_{j,k=1}^n \int\limits_0^{T/2} 2^{2q}\|\partial_t v_q\|_{L^2(\R_x^n)} \mu(\eps)\|v_q\|_{L^2(\R_x^n)} dt \\
&& \qquad \leq \int\limits_0^{T/2} \|\partial_t v_q\|_{L^2(\R_x^n)}^2 dt + C_2 2^{4q} \mu(\eps) \int\limits_0^{T/2} \|v_q\|_{L^2(\R_x^n)}^2 dt,
\end{eqnarray*} where we used (see \eqref{approx1}) \begin{eqnarray*}
&&\|\left(S_{q-1}(a_{jk}(t,x)-a_{jk}^\eps(t,x)\right) \partial_{x_k}v_q\|_{L^2(\R_x^n)} \\
&& \qquad \qquad \qquad \leq \|\left(S_{q-1}(a_{jk}(t,x)-a_{jk}^\eps(t,x)\right)\|_{L^\infty(\R_x^n)} \|\partial_{x_k}v_q\|_{L^2(\R_x^n)} \\
&& \qquad \qquad \qquad \leq C2^q \mu(\eps) \|v_q\|_{L^2(\R_x^n)}
\end{eqnarray*}
and the fact that $\mu^2(\varepsilon)\leq \mu(1) \mu(\varepsilon)$; remark that here the constant $C_2$ depends only on $\mu$ and on $\max_{j,k} \|a_{j,k}\|_{C^\mu([0,T], L^\infty(\R^n_x))}$.

Resuming, we have
\begin{eqnarray}
\nonumber && \int\limits_0^{T/2} \Big\| \partial_t v_q + \sum\limits_{j,k=1}^n \partial_{x_j}\left(\Delta_q (a_{jk}(t,x)\partial_{x_k} v)\right) + \Phi'(\gamma(T-t))v_q \Big\|^2_{L^2(\R^n_x)} dt \\
\nonumber && \quad \geq \int\limits_0^{T/2} \Big\|  \sum\limits_{j,k=1}^n \partial_{x_j}\left(S_{q-1}a_{jk}(t,x)\partial_{x_k}v_q\right) + \Phi'(\gamma(T-t))v_q \Big\|^2_{L^2(\R^n_x)} dt \\
 && \qquad + \gamma\int\limits_0^{T/2} \Phi''(\gamma(T-t))\|v_q\|_{L^2(\R^n_x)}^2 dt - C_3( \frac{\mu(\eps)}{\eps}2^{2q}+2^{4q} \mu(\eps)) \int\limits_0^{T/2} \|v_q\|_{L^2(\R_x^n)}^2 dt, \label{final1} 
\end{eqnarray} where $C_3$ depends only on $\max_{j,k} \|a_{j,k}\|_{C^\mu([0,T], L^\infty(\R^n_x))}$.

%%%%%%%%%%%%%%%%%%%%%%%%%%%%
\subsection{End of the proof: high frequencies.}%%%%%
%%%%%%%%%%%%%%%%%%%%%%%%%%%%%

We detail the end of the proof, starting with the high frequencies. We follow the lines of \cite{DSP, DSP2012}. By Remark \ref{Pos} there exist $q_0 \geq -1$ and a constant $C_4>0$ such that 
\begin{eqnarray*}
&& \big\| \sum\limits_{j,k=1}^n \partial_{x_j}\left(S_{q-1}a_{jk}(t,x)\partial_{x_k}v_q\right) \big\|_{L^2(\R^n_x)} \|v_q\|_{L^2(\R^n_x)} \\
&& \qquad \geq  \big| \Big\langle \sum\limits_{j,k=1}^n \partial_{x_j}\left(S_{q-1}a_{jk}(t,x)\partial_{x_k}v_q\right) \, |\, v_q \Big\rangle_{L^2(\R^n_x)}  \big| \geq  \frac{a_0}{2} \|\nabla_x v_q\|_{L^2(\R^n_x)}^2 \geq C_4 a_0 2^{2q}\|v_q\|_{L^2(\R^n_x)}^2,
\end{eqnarray*}
 where $a_0$ is the constant in \eqref{elliptic}.

Suppose first that $\Phi'(\gamma(T-t)) \leq \frac{1}{2}C_4 a_0 2^{2q}$. Then, from the last inequality, we deduce 
\begin{equation*}
\big\| \sum\limits_{j,k=1}^n \partial_{x_j}\left(S_{q-1}a_{jk}(t,x)\partial_{x_k}v_q\right) \big\|_{L^2(\R^n_x)} - \Phi'(\gamma(T-t))\|v_q\|_{L^2(\R^n_x)} \geq \frac{1}{2}C_4 a_0 2^{2q}.
\end{equation*} 
We choose $\eps = 2^{-2q}$ in such a way that the quantities $2^{4q}\mu(\eps)$ and $\displaystyle{2^{2q}\frac{\mu(\eps)}{\eps}}$ are equal.
Using the fact that $\Phi''(\gamma(T-t)) \geq 1$ (this is a consequence of the nonrestrictive hypothesis that $\mu(1)=1$; if it is not so, the modifications of the subsequent lines are easy), we obtain that 
\begin{eqnarray*}
&& \int\limits_0^{T/2} \Big\| \partial_t v_q + \sum\limits_{j,k=1}^n \partial_{x_j}\left(\Delta_q (a_{jk}(t,x)\partial_{x_k} v)\right) + \Phi'(\gamma(T-t))v_q \Big\|^2_{L^2(\R^n_x)} dt \\
&&\quad\geq \int\limits_0^{T/2} \Big( \big\| \sum\limits_{j,k=1}^n \partial_{x_j}\left(S_{q-1}a_{jk}(t,x)\partial_{x_k}v_q\right) \big\|_{L^2(\R^n_x)} - \Phi'(\gamma(T-t))\|v_q\|_{L^2(\R^n_x)} \Big)^2 \\
&& \qquad + \gamma \int\limits_0^{T/2}\Phi''(\gamma(T-t))\|v_q\|_{L^2(\R^n_x)}^2 dt - 2C_3 2^{4q}\mu(2^{-2q})\int\limits_0^{T/2}\|v_q\|_{L^2(\R^n_x)}^2 dt\\
&& \qquad \geq \int\limits_0^{T/2} \Big( (\frac{1}{2}C_4 a_0)^2 2^{4q} + \gamma - 2C_32^{4q}\mu(2^{-2q}) \Big)\|v_q\|_{L^2(\R^n_x)}^2 dt \\
&& \qquad \quad\geq \int\limits_0^{T/2} \Big( \Big( \frac{1}{2}(\frac{1}{2}C_4 a_0)^2 - 2C_3(\mu(2^{-2q}))\Big) 2^{4q} + \frac{\gamma}{3}\Big)  \|v_q\|_{L^2(\R^n_x)}^2 dt \\
&&  \qquad\qquad + \int\limits_0^{T/2} \Big( \frac{1}{2}(\frac{1}{2}C_4 a_0)^2 2^{4q}+ \frac{2}{3}\gamma \Big) \|v_q\|_{L^2(\R^n_x)}^2 dt.
\end{eqnarray*} 
Since we have $\lim_{q \rightarrow +\infty} \mu(2^{-2q}) = 0$, there exists an $\gamma_0>0$ such that \begin{eqnarray*}
\Big( \frac{1}{2}(\frac{1}{2}C_4 a_0)^2 - 2C_3)\mu(2^{-2q}))\Big) 2^{4q} + \frac{\gamma}{3} \geq 0
\end{eqnarray*} for $\gamma \geq \gamma_0$ and all $q \geq q_0$. 
Consequently, for $\gamma \geq \gamma_0$,
\begin{eqnarray*}
&& \int\limits_0^{T/2} \Big\| \partial_t v_q + \sum\limits_{j,k=1}^n \partial_{x_j}\left(\Delta_q (a_{jk}(t,x)\partial_{x_k} v)\right) + \Phi'(\gamma(T-t))v_q \Big\|^2_{L^2(\R^n_x)} dt \\
 &&\qquad\qquad\geq \int\limits_0^{T/2} \Big( \frac{1}{2}(\frac{1}{2}C_4 a_0)^2 2^{4q}+ \frac{2}{3}\gamma \Big) \|v_q\|_{L^2(\R^n_x)}^2 dt.
\end{eqnarray*} 
Recall now \eqref{TechCond2}. Using it with $s=1/2$, we have that three exists $C_0>
0$ such that, for all $q\geq-1$, we have $\mu(2^{-2q})\leq C_02^{-q}$. Then, for all $q\geq-1$ and for all $\gamma \geq \gamma_0$,
\begin{equation*}
 \frac{1}{2}(\frac{1}{2}C_4 a_0)^2 2^{4q}+ \frac{1}{6}\gamma\geq C_5\gamma^{1\over 4}2^{3q}\geq C_6\gamma^{1\over 4}2^{4q}\mu(2^{-2q}),
 \end{equation*}
 for some $C_5, \ C_6>0$. Finally 
\begin{eqnarray}
\nonumber&& \int\limits_0^{T/2} \Big\| \partial_t v_q + \sum\limits_{j,k=1}^n \partial_{x_j}\left(\Delta_q (a_{jk}(t,x)\partial_{x_k} v)\right) + \Phi'(\gamma(T-t))v_q \Big\|^2_{L^2(\R^n_x)} dt \\
 &&\qquad\qquad\geq \int\limits_0^{T/2} \Big( {\gamma\over 2}+C_6\gamma^{1\over 4}2^{4q}\mu(2^{-2q}) \Big) \|v_q\|_{L^2(\R^n_x)}^2 dt.\label{finalest1}
\end{eqnarray} 

Suppose  now $\Phi'(\gamma(T-t)) \geq \frac{1}{2}C_4 a_0 2^{2q}$. Again we choose $\eps = 2^{-2q}$. Then, using \eqref{ODE}, the fact that $a_0 \leq 1$ and the properties of $\mu$, we get
\begin{eqnarray*}
\Phi''(\gamma(T-t)) &=& (\Phi'(\gamma(T-t)))^2 \mu\left( \frac{1}{\Phi'(\gamma(T-t))} \right) \\
&\geq& (\frac{1}{2}C_4 a_0)^2 2^{4q} \mu\left( \frac{2}{C_4 a_0} 2^{-2q} \right) \geq (\frac{1}{2}C_4 a_0)^2 2^{4q} \mu(2^{-2q}).
\end{eqnarray*} 
Hence there exist $\gamma_0$ and constants $C_7, C_8>0$  such that, for $\gamma \geq \gamma_0$, 
\begin{eqnarray}
\nonumber&& \int\limits_0^{T/2} \Big\| \partial_t v_q + \sum\limits_{j,k=1}^n \partial_{x_j}\left(\Delta_q (a_{jk}(t,x)\partial_{x_k} v)\right) + \Phi'(\gamma(T-t))v_q \Big\|^2_{L^2(\R^n_x)} dt \\
\nonumber&&\quad\geq \int\limits_0^{T/2} \Big( \big\| \sum\limits_{j,k=1}^n \partial_{x_j}\left(S_{q-1}a_{jk}(t,x)\partial_{x_k}v_q\right) \big\|_{L^2(\R^n_x)} - \Phi'(\gamma(T-t))\|v_q\|_{L^2(\R^n_x)} \Big)^2 \\
\nonumber&& \qquad + \gamma \int\limits_0^{T/2}\Phi''(\gamma(T-t))\|v_q\|_{L^2(\R^n_x)}^2 dt - 2C_3 2^{4q}\mu(2^{-2q})\int\limits_0^{T/2}\|v_q\|_{L^2(\R^n_x)}^2 dt\\
\nonumber&& \qquad\quad \geq \int\limits_0^{T/2} \Big( \frac{\gamma}{2} + \big(\frac{\gamma}{2}(\frac{1}{2}C_4 a_0)^2-2C_3\big) 2^{4q}\mu(2^{-2q} )\Big)\|v_q\|_{L^2(\R^n_x)}^2 dt \\
\nonumber&& \qquad \qquad \geq \int\limits_0^{T/2} \Big( \frac{\gamma}{2} + C_7  \gamma 2^{4q}\mu(2^{-2q}) \Big)\|v_q\|_{L^2(\R^n_x)}^2 dt  \\
&& \qquad \qquad\quad  \geq \int\limits_0^{T/2} \Big( \frac{\gamma}{2} + C_8  \gamma^{1\over 4} 2^{4q}\mu(2^{-2q}) \Big)\|v_q\|_{L^2(\R^n_x)}^2dt  .\label{finalest2}
\end{eqnarray} 
Recall now that $2^{2q}\mu(2^{-2q})=2^{2q}\omega^2(2^{-q})=\Omega^2(q) $. From \eqref{finalest1} and \eqref{finalest2}
we immediately obtain
\begin{eqnarray}
\nonumber &&\sum\limits_{q \geq q_0} 2^{-2sq} \int\limits_0^{T/2} \Big\| \partial_t v_q + \sum\limits_{j,k=1}^n \partial_{x_j}\left(S_{q-1}a_{jk}(t,x)\partial_{x_k}v_q\right) + \Phi'(\gamma(T-t))v_q \Big\|^2_{L^2(\R^n_x)} dt \\
&&\label{finalissima1}\qquad\geq\sum\limits_{q \geq q_0} 2^{-2sq} \int\limits_0^{T/2} \Big( \frac{\gamma}{2} + C  \gamma^{1\over 4} \Omega^2(q) 2^{2q} \Big)\|v_q\|_{L^2(\R^n_x)}^2 dt.
\end{eqnarray}

%%%%%%%%%%%%%%%%%%%%%%%%%%%%
\subsection{End of the proof: low frequencies.}%%%%%
%%%%%%%%%%%%%%%%%%%%%%%%%%%

In this section we complete the proof for low frequencies. We sum \eqref{final1} multiplied with $2^{-2qs}$ for $q \leq q_0-1$ ($q_0$ is the same as in the previous section).
We set $\eps = 2^{-2q_0}$ and we obtain 
\begin{eqnarray*}
&& \sum\limits_{q \leq q_0-1} 2^{-2sq} \int\limits_0^{T/2} \Big\| \partial_t v_q + \sum\limits_{j,k=1}^n \partial_{x_j}\left(S_{q-1}a_{jk}(t,x)\partial_{x_k}v_q\right) + \Phi'(\gamma(T-t))v_q \Big\|^2_{L^2(\R^n_x)} dt \\
&& \qquad \geq \sum\limits_{q \leq q_0-1} 2^{-2sq} \int\limits_0^{T/2} \Big( \gamma- 2C_3 \mu(2^{-2q_0}) 2^{2(q+q_0)} \Big) \|v_q\|_{L^2(\R_x^n)}^2 dt.\\
\end{eqnarray*} 
Taking $\gamma_0$ large enough we can absorb the negative term. We easily obtain 
\begin{eqnarray}
&&\nonumber \sum\limits_{q \leq q_0-1} 2^{-2sq} \int\limits_0^{T/2} \Big\| \partial_t v_q + \sum\limits_{j,k=1}^n \partial_{x_j}\left(S_{q-1}a_{jk}(t,x)\partial_{x_k}v_q\right) + \Phi'(\gamma(T-t))v_q \Big\|^2_{L^2(\R^n_x)} dt \\
&&\label{finalissima2} \qquad \geq\sum\limits_{q \leq q_0-1} 2^{-2sq} \int\limits_0^{T/2} \Big( \frac{\gamma}{2} + C  \gamma^{1\over 4} \Omega^2(q) 2^{2q} \Big)\|v_q\|_{L^2(\R^n_x)}^2 dt.
\end{eqnarray} 
Summing \eqref{finalissima1} and \eqref{finalissima2} we obtain \eqref{eq:CarlemanTrans}. The proof is completed.

%%%%%%%%%%%%%%%%%%%%%%%%%%%%%%%%%%%%%%%%%%
%%%%%%%%%%%%%%%%%%%%%%%%%%%%%%%%%%%%
\subsubsection*{Acknowledgements}%%%%%%%%%%%%%%%
%%%%%%%%%%%%%%%%%%%%%%%%%%%%%%%%%%%%%
%%%%%%%%%%%%%%%%%%%%%%%%%%%%%%%%%%

The first author is member of the Gruppo Nazionale per l'Analisi Matematica, la Probabilit\`a
e le loro Applicazioni (GNAMPA) of the Istituto Nazionale di Alta Matematica (INdAM).

A part of this work was done while the third author was visiting the Department of Mathematics and Geosciences of Trieste University with the support of GNAMPA--INdAM.

\vskip2cm
\noindent
Daniele Del Santo\\
Dipartimento di Matematica e Geoscienze\\
Universit\`a di Trieste\\
Via Valerio 12/1, I-34127 Trieste, Italy\\
{\tt delsanto@units.it}

\vskip1cm
\noindent
Christian J\"ah\\
Institut f\"ur Angewandte Analysis\\
Fakult\"at f\"ur Mathematik und Informatik\\
Technische Universit\"at Bergakademie Freiberg\\
Pr\"uferstrasse 9, D-09596 Freiberg, Germany\\
{\tt christian.jaeh@math.tu-freiberg.de}

\vskip1cm
\noindent
Marius Paicu\\
Institut de Math\'ematiques de Bordeaux\\
Universit\'e Bordeaux 1\\
351, cours de la Lib\'eration, F-33405 Talence cedex, France\\
{\tt Marius.Paicu@math.u-bordeaux1.fr}

\end{document}